\documentclass[11pt]{article}

\usepackage[left=3cm,right=3cm, top=2.5cm,bottom=2.5cm,bindingoffset=0cm]{geometry}

\usepackage[numbers, sort&compress]{natbib}
 \usepackage{tikz-cd}
\usepackage{amsmath,amssymb, amsthm}
\usepackage{mathtools}

\usetikzlibrary{calc,intersections,through,backgrounds}

\usepackage[nottoc]{tocbibind}

\newtheorem{lemma}{Lemma}[section]
\newtheorem{proposition}[lemma]{Proposition}
\newtheorem{theorem}[lemma]{Theorem}
\newtheorem{alphabetictheorem}{Theorem}

\newtheorem{corollary}[lemma]{Corollary}
\newtheorem{assumption}[]{Assumption}

\usepackage{mathabx}
\def\acts{\mathrel{\reflectbox{$\righttoleftarrow$}}}

\usepackage{hyperref}

\theoremstyle{definition}

\newtheorem{alphabeticexample}[alphabetictheorem]{Example}

\newtheorem{definition}[lemma]{Definition}
\newtheorem{remark}[lemma]{Remark}

\mathtoolsset{showonlyrefs}                           

\newcommand{\Ker}[1]{\mathrm{Ker} \, #1}
\newcommand{\Coker}[1]{\mathrm{Coker} \, #1}
\newcommand{\rank}[1]{\mathrm{rank} \, #1}
\newcommand{\codim}[1]{\mathrm{codim} \, #1}

\newcommand{\diff}[1]{{d}  #1}

\newcommand{\Z}{\mathbb{Z}}
\newcommand{\N}{\mathbb{N}}

\renewcommand{\L}{{L}}

\newcommand{\R}{\mathbb{R}}
\newcommand{\Complex}{\mathbb{C}}

\newcommand{\T}{{T}}
\renewcommand{\N}{{N}}
\newcommand{\Cont}{{C}}

\newcommand{\tr}[1]{\mathrm{tr} \, #1}

\newcommand{\GL}{\mathrm{GL}}

\newcommand{\sgrad}[1]{X_{#1}}
\newcommand{\F}{\mathcal{F}}
\newcommand{\Diff}{\mathrm{Diff}}
\newcommand{\ADiff}{\mathrm{LDiff}}
\newcommand{\id}{\mathrm{id}}
\renewcommand{\sp}{\mathfrak{sp}}
\renewcommand{\i}{\sqrt{-1}}

\newcounter{ab}

\date{}
\title{Smooth invariants of focus-focus singularities and obstructions to product decomposition}
\author{Alexey Bolsinov\thanks{Department of Mathematical Sciences, Loughborough University, and  Faculty of Mechanics and Mathematics, Moscow State University.  E-mail: {\tt A.Bolsinov@lboro.ac.uk}} \,  
 and Anton Izosimov\thanks{
Department of Mathematics,
University of Arizona.
E-mail: {\tt izosimov@math.arizona.edu}
}
}
\begin{document}
\maketitle
\begin{abstract}{We study focus-focus singularities (also known as nodal singularities, or pinched tori) of Lagrangian fibrations on symplectic $4$-manifolds. We show that, in contrast to elliptic and hyperbolic singularities, there exist homeomorphic focus-focus singularities which are not diffeomorphic. Furthermore, we obtain an algebraic description of the moduli space of focus-focus singularities up to smooth equivalence, and show that for double pinched tori this space is one-dimensional. Finally, we apply our construction to disprove Zung's conjecture which says that any non-degenerate singularity can be {smoothly} decomposed into an almost direct product of standard singularities.}\end{abstract}
\tableofcontents

\section{Introduction}

The main goal of the present paper  is to study one interesting property of focus-focus singularities  (also known as nodal singularities, or pinched tori) in the context of the theory of singular Lagrangian fibrations or, which is essentially the same, in the context of topology of finite-dimensional integrable Hamiltonian systems.   

From the viewpoint of symplectic topology,  an integrable system on a symplectic manifold $(M^{2n}, \omega)$ is defined by a collection of Poisson commuting functions $f_1, f_2, \dots, f_n \colon M^{2n} \to \R$ which are independent almost everywhere on $M^{2n}$.
Throughout the paper we assume that the corresponding moment map $\mathcal F=(f_1,\dots, f_n)\colon M^{2n} \to \R^n$ is proper. In particular, the Hamiltonian flows generated by $f_1, f_2, \dots, f_n$ are all complete so that  $M^{2n}$ is endowed with the natural $\R^n$-action generated by these flows. The fibers of the singular Lagrangian fibration on $M^{2n}$, associated with this integrable system, are connected components of  level sets $\mathcal F^{-1}(a)$, $a\in \R^n$.   According to the Arnold-Liouville theorem,  regular fibers are Lagrangian tori of dimension $n$.  Here, however, we are mainly interested in singular fibers containing those points $P\in M^{2n}$ where $\rank \diff\mathcal F (P)<n$.  

In the case of non-degenerate singularities, topological description of singular fibers in the semiglobal setting is due to N.T.\,Zung \cite{AL}. His fundamental decomposition theorem states that under some mild additional conditions, such a singularity is homeomorphic to an almost direct product of elementary bricks of four types:  regular, elliptic, hyperbolic, and focus-focus. The latter case is of particular interest as focus-focus singularities possess a number of remarkable properties and have far reaching applications in symplectic geometry, see e.g.~\cite{kontsevich2006affine, leung2010almost, bernard2009lagrangian}. Among numerous works on focus-focus singularities we would like to emphasize, first of all, the papers by  V.\,Matveev~\cite{Matveev} and N.T.\,Zung \cite{ZungFF, ZungFF2}  (topological classification), as well as  
S.\,V\~{u} Ng\d{o}c \cite{SanFF}  (symplectic classification). 

The properties and invariants we are going to discuss in this paper are related to the following phenomenon: unlike elliptic and hyperbolic case, there exist homeomorphic focus-focus singularities which are  not diffeomorphic.  In other words, in the focus-focus case there are non-trivial smooth invariants, somewhere between topological and symplectic ones previously studied.  This phenomenon was first noticed in \cite[Section 9.8.2]{intsys}.

 Our motivation to study smooth invariants of (not necessarily non-degenerate) singular Lagrangian fibrations comes from symplectic geometry.  Of course,  our primary goal is to classify such fibrations up to symplectomorphisms.   However, if we are looking for a symplectic map between two Lagrangian fibrations  $\F_i \colon M_i   \to B_i$, it is quite natural to do it in two steps.  First, we find a fiberwise diffeomorphism  $\Psi : M_1 \to M_2$. As a result we obtain two different symplectic forms on $M_1$,   the original one $\omega_1$ and the pullback $\omega':=\Psi^*\omega_2$,  such that the fibration given on $M_1$ is Lagrangian with respect to both of them. After this we can try to find another map $\Psi' : M_1 \to M_1$ such that each fiber is preserved and $\omega_1 = {\Psi' }^*\omega'$.   Working in this setting is more convenient for many reasons, for instance,  to ``compare''  two different symplectic forms on the same manifold we can use the usual Moser trick which can be naturally adapted to Lagrangian fibrations.

Recall that the topology of a focus-focus singularity  (for an integrable system on a symplectic $4$-manifold) is completely determined by the number of focus-focus points on the singular fiber. In particular,  if the singular fiber   contains $n$ focus-focus critical points  (and no other critical points!), then it is an $n$-pinched torus illustrated in Figure \ref{fig1}  (see more detailed description in Section \ref{sec:definitions}).

\def\nff{5}
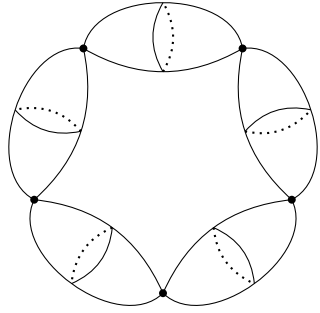
\begin{figure}[t]
		\centerline{
		\begin{tikzpicture}[thick, scale = 1.5, rotate = 90]
\foreach \a in {1,2,...,\nff}{
  \fill (\a*360/\nff + 180/\nff: 1.2cm) circle [radius=1pt];
  \draw [thin, name path = IN] (\a*360/\nff + 180/\nff: 1.2cm) to [bend left](\a*360/\nff + 360/\nff + 180/\nff: 1.2cm);
   \draw [thin, name path = OUT]  (\a*360/\nff + 180/\nff: 1.2cm)  to [bend right = 80](\a*360/\nff + 360/\nff + 180/\nff: 1.2cm);
   \path [name path = R]  (0,0) to (\a*360/\nff + 360/\nff: 1.5cm);
   \path [name intersections={of=IN and R,by=IIP}];
      \path [name intersections={of=OUT and R,by=OIP}];
       \draw [thin](IIP) to [bend left](OIP);
         \draw [dotted](IIP) to [bend right](OIP);
}
\end{tikzpicture}
}
			\caption{{\small focus-focus singularity with $n = \nff$ critical points}}\label{fig1}
	\end{figure}

In \cite{intsys}, there is only a short remark about existence of non-trivial smooth invariants starting from $n=2$  (for $n=1$,  all focus-focus singularities with one pinched point on the fiber are diffeomorphic). However no explanation of their nature is given.  This paper is aimed at filling this gap.  The description of smooth invariants for $n\ge 2$ will be given in Section~\ref{sec:smooth}. In brief, an $n$-pinched focus-focus singularity is determined (up to diffeomorphisms) by $n-1$ \textit{gluing maps} $\phi_{1,2}, \dots, \phi_{1,n}$ which prescribe how standard neighborhoods of $n$ focus-focus points are ``glued'' together. These maps can be interpreted as elements of  the group $G$ of germs at $z=0$ of local (real) diffeomorphisms of $\mathbb C$ fixing the origin. These diffeomorphisms are defined not uniquely, but only up to an action of the subgroup $H$ of \textit{liftable} germs  (Definition~\ref{def:liftable}) that consists of germs divisible by $z$ or $\bar z$. Therefore, the space of smooth structures on an $n$-pinched focus-focus singularity can be thought of as the quotient space of $G^{n-1}$ by the corresponding action of liftable germs. More precisely, we have the following result.
%
%
 \begin{alphabetictheorem}[=Theorem \ref{thm:thm1}]
 Two focus-focus singularities with $n$ pinch points are fiberwise diffeomorphic if and only if the corresponding gluing maps $\phi_{1,2}, \dots, \phi_{1,n}$  and $\tilde \phi_{1,2}, \dots, \tilde \phi_{1,n}$  are related by
\begin{equation}
\label{action}
 \tilde \phi_{1,i} = \psi_1\circ\phi_{1,i}\circ \psi_i^{-1},
\end{equation}
where $\psi_1, \dots, \psi_n \in H$ are liftable.
 
 In other words, smooth structures on an $n$-pinched focus-focus singularity are in one-to-one correspondence 
 with the orbits of the action of $H^n$ on $G^{n-1}$ defined by \eqref{action}.
The germ groups $G$ and  $H$ can be replaced by the corresponding groups of infinite jets.
 \end{alphabetictheorem}

Since the groups $G$ and $H$ are infinite-dimensional,  a complete description of $C^\infty$-smooth invariants, i.e. invariants of action \eqref{action}, is a non-trivial problem.  One can, however, describe $C^k$-invariants, i.e. invariants of the same action \eqref{action} with germ groups $G$ and $H$ replaced by the corresponding groups of $k$-jets. For instance, in the simplest case $k=1$,  the group $G$  is isomorphic to $\GL_2(\R)$, and  $H \subset \GL_2(\R)$ consists of $\mathbb C$-linear  and $\mathbb C$-antilinear functions.  The orbits of the corresponding action \eqref{action} are easy to describe:
 \begin{alphabetictheorem}[=Theorem \ref{theorem:firstOrderInv} + Proposition \ref{prop3.4}]\label{thmb}
Focus-focus singularities with $n$ pinch points have $2n-3$ $C^1$-invariants. These invariants are given,
in terms of the corresponding Poisson-commuting Hamiltonians, by $n-1$ complex numbers
   \begin{equation}
        \mu_i = \frac{ \lambda_i - \lambda_1}{ { \lambda}_i + \bar \lambda_1},
    \end{equation}
    considered up to multiplication by the same complex number of absolute value $1$ and simultaneous complex conjugation.
Here $\lambda_i$ is a (suitably chosen) eigenvalue a generic Hamiltonian linearized at $i$'th focus-focus point. 
 \end{alphabetictheorem}
 In Section \ref{sec:foics} we also give a geometric interpretation of these invariants in terms of complex structures on the base of a focus-focus fibration.

 Although one can similarly define and describe $C^k$-invariants for every $k$, for $n=2$ pinch points this is unnecessary. In this case the $C^1$-invariant (which is unique since $2n-3=1$) already separates generic orbits of action \eqref{action}, i.e. generic orbits are of codimension one:
\begin{alphabetictheorem}[=Theorem \ref{thm:c1tocinfty}]
The (regular part of the) space of double-pinched focus-focus singularities, considered up to $C^\infty$-diffeomorphisms, is one-dimensional and parametrized by the $C^1$-invariant.
\end{alphabetictheorem}
For general $n$  (number of pinch points on the singular fiber)  there is, apparently, a similar stabilization phenomenon. We conjecture that for each $n$ there is $k = k(n)$ such that $C^k$-invariants  allow one to distinguish generic orbits  of action \eqref{action}, so that the codimension of generic orbits (or equivalently, the number of $C^\infty$-smooth invariants) is finite (see Remark \ref{rem:npinchedstab}).
\par
It is important to note that since the underlying symplectic structure does not play any essential role in the context of smooth invariants,  one can consider a  (potentially) more general situation of toric fibrations with an isolated focus-like fiber.  (Such fibrations arise, in particular, in integrable non-Hamiltonian systems. It is shown in  \cite{ZungNonHam} that focus-like singularities are typical in this general setting, and they indeed naturally appear in integrable nonholonomic systems \cite{BolsFocus, Cushman}.) That is what we are actually doing in our paper. However, this approach naturally raises the symplectization problem: is it true that any focus-like singularity can be endowed with a suitable symplectic structure in such a way that all the fibers become Lagrangian? A positive answer is given in Section \ref{sec:symp1}. 

Finally,  by using non-triviality of smooth invariants in the focus-focus case, we disprove the conjecture stated by Zung in  \cite{AL}. 
Recall that his main result is that any non-degenerate singularity can be {\it topologically} decomposed into an almost direct product of ``elementary bricks'' of  dimension 2 or 4.   In other words,  every such singularity is {\it homeomorphic} to the quotient of a certain direct product of ``elementary bricks'' by a symplectic component-wise fibration-preseving action of a finite group. It is quite obvious that such a decomposition is in general {\it not symplectic}, i.e. no suitable symplectomorphism can be found.   However,  Zung conjectured that this decomposition is {\it smooth}.  
In Section \ref{sec:example} we construct the following counterexample:
\begin{alphabeticexample}[see Section \ref{sec:example}]
Consider a one-parametric family of double-pinched focus-focus singularities on a symplectic $4$-manifold such that the $C^1$-invariant defined in Theorem \ref{thmb} varies within the family. Multiplying the total space of this family by a circle $S^1$, one can turn it into a Lagrangian fibration on a $6$-manifold, which has a non-degenerate singularity (focus singularity of rank one with two critical circles on each singular fiber). It is easy to see that the so-obtained singularity is \textit{homeomorphic}  to the direct product of a focus-focus singularity with two pinch points and a regular circle fibration  on an annulus.  However, as we show in Section \ref{sec:example}, this singularity is not \textit{diffeomorphic} to any product of this kind. The proof is essentially based on  the interpretation of $C^1$-invariants in terms of complex structures presented in Section \ref{sec:foics}.
\end{alphabeticexample}

\medskip

{\bf Acknowledgements.}   The work of the first author was supported by the Russian Science Foundation (grant No. 17-11-01303). The second author acknowledges the hospitality of Max Planck Institute for Mathematics, Bonn, where a part of this work was done. The authors are grateful to Gleb Smirnov for fruitful discussions and to the referee of this paper for his comments and remarks.

\section{Basic definitions and facts on focus-focus singularities}\label{sec:definitions}

Consider two commuting functions $H, F$ on a symplectic $4$-manifold $(M, \omega)$.    Let $P\in M$  be a rank $0$ singular point of the moment map
$\mathcal F = (H, F): M^4 \to \mathbb R^2$, which means that   $dH(P)=0$ and $dF(P)=0$.   

Consider the linearizations  $A_{H}$, $A_{F}$ of the Hamiltonian vector fields  $\sgrad H := \omega^{-1}(dH)$, $\sgrad F:= \omega^{-1}(dF)$   at the singular point $P$.  Since  $H$ and $F$ commute and the vector fields  $\sgrad {H}$, $\sgrad {F}$ are Hamiltonian, the operators $A_{H}$ and $A_{F}$  can be understood as commuting elements of the symplectic Lie algebra  $\sp(\T_P M, \omega)$.  Recall that  $P\in M$ is called {\it non-degenerate} if the commutative subalgebra in  $\sp(\T_P M, \omega)$ generated by the operators $A_{H}$ and  $A_{F}$ is a Cartan subalgebra.  Singular points of focus-focus type  are defined by the following additional condition.

\begin{definition}
A non-degenerate singular point $P\in M$ is said to be of \textit{focus-focus type} if the corresponding Cartan subalgebra is conjugate to the subalgebra of the form   
    \begin{align*}
      \begin{pmatrix}
            a & -b &  0 &  0\\
            b &  a &  0 &  0\\
            0 &  0 & -a & -b\\
            0 &  0 &  b & -a
        \end{pmatrix}, \quad  a,b \in \R.
    \end{align*}
Here we use the standard matrix representation of $\sp(\T_P M, \omega)$ with $\omega = dp_1\wedge dq_1 +  dp_2\wedge dq_2$  and the coordinates ordered as $p_1, p_2, q_1, q_2$.
\end{definition}

Equivalently,  one can say that  for a generic linear combination $\alpha H + \beta F$, the operator  $\alpha A_{H} + \beta  A_{F}$  is diagonalizable,  its eigenvalues are distinct and form a complex quadruple $\pm a\pm\i b$. 

It is easy to see that a focus-focus is an isolated singular point of the moment map $\mathcal F=(H,F)$. 

According to Eliasson's theorem  (see \cite{Vey, Eliasson, miranda2005singular} for the general case and also \cite{Chaperon, Wacheux} for the focus-focus case),   locally every non-degenerate singularity can be reduced to a standard normal form. In the case of a focus-focus singularity we have 

\begin{proposition}\label{prop:norm}
   In a neighborhood of a non-degenerate singular point of focus-focus type, there is a symplectic coordinate system  $p_1$, $p_2$, $q_1$, $q_2$,
   in which the commuting functions  $H$ and $F$ take the following form:
    \begin{align}
        \begin{cases}\label{map}
            H = H(f_1, f_2)\\
            F = F(f_1, f_2)\\
            f_1 := p_1 q_1 + p_2 q_2\\
            f_2 := p_1 q_2 - q_1 p_2,
        \end{cases}
    \end{align}
    Moreover, the transition map $f_1, f_2 \mapsto H,F$ is non-degenerate  (i.e., a local diffeomorphism).
\end{proposition}

For our purposes, it will be more convenient to rewrite the above formulas in complex notation.  Namely, we set 
$$
u:=p_1 - \i p_2, \  v:=q_1 + \i q_2.
$$
Then
$$
uv =  f_1(u,v) + \i f_2(u,v), \quad \omega = \mathrm{Re}\, (du\wedge dv).
$$

So, a singular point $P$ of the moment map $\mathcal F: M^4 \to \R^2$ is of focus-focus type if there exist local complex coordinates $(u,v)$ on $M^4$ (canonical in the  sense that $\mathrm{Re}\, (du\wedge dv) = \omega$) and a local complex coordinate $z$ on $\R^2$ in which $\mathcal F$ takes the form $z=uv$. (Here by complex coordinates we mean a $\Cont^\infty$-smooth, but not necessarily holomorphic, map to a complex vector space of appropriate dimension.)


These  normal coordinates immediately give us a local description of the corresponding Lagrangian fibration in a neighborhood of a focus-focus point.
\begin{proposition}\label{prop:local}
    Consider the neighborhood $U$ of a focus-focus point which is a ball in normal coordinates  {\rm(}that is $U:= \{ p_1^2 + p_2^2 + q_1^2 + q_2^2  < \varepsilon\} = \{ |u|^2 + |v|^2 < \varepsilon\}${\rm)}.  Let  $\L_\delta := \{ (u,v)\in U~|~ uv=\delta \in \mathbb C^* \}$ be the intersection of a regular fiber {\rm(}that is sufficiently close to the singular one $\L_0:=\{ (u,v)\in U~|~ uv=0\}${\rm)} with the neighborhood $U$. Then 
      \begin{enumerate}
        \item $\L_\delta$ is diffeomorphic to a cylinder.
        \item The Hamiltonian flow of  $f_2=\mathrm{Im}\, uv$ is $2\pi$-periodic {\rm(}i.e., defines a Hamiltonian $S^1$-action{\rm)}.  Every trajectory of this flow generates the first homology group  $H_1(\L_\delta, \Z)$.
        \item The singular fiber $\L_0$ is the union of two transversally intersecting discs.
    \end{enumerate}
\end{proposition}

\begin{remark}\label{rem:uniqueness}
Notice that the functions  $f_1$ and $f_2$ can be obtained from the original commuting functions $H$ and $F$ by means of a non-degenerate change of variables (i.e., by a suitable local diffeomorphism $f_1=f_1(H,F),  f_2=f_2(H,F)$).  The first of them $f_1$ is defined  uniquely up to sign and adding a flat function, while the second one $f_2$ is defined up to sign, see \cite{SanFF}.  \par
\end{remark}
The next theorem, due to Matveev \cite{Matveev} and Zung \cite{ZungFF}, describes focus-focus singularities in the semi-local setting, i.e. in a neighborhood of the singular fiber.  
\begin{theorem}\label{th:focus}     Let $\mathcal F: M^4 \to \R^2$ be a moment map defined by two Poisson commuting functions. Let also $\L_0:=\F^{-1}(0)$  be a singular fiber that contains a focus-focus point $P$.  Assume that the singular fiber $\L_0$ is compact, and that all singular points of  $\mathcal F$ on $\L_0$ are non-degenerate of rank $0$.
    Then 
     \begin{enumerate}
        \item  All singular points on  $\L_0$ are of focus-focus type and there are finitely many of them. 
        \item The singular fiber $\L_0$ is the union of $n$ Lagrangian spheres transversally intersecting at singular points\footnote{If $n=1$, the fiber is an immersed Lagrangian sphere with one self-intersection point.}, where $n$ is the number of singular points on the fiber (see Figure~\ref{fig1}).          
        \item  A sufficiently small neighborhood $U(\L_0):=\mathcal F^{-1}(B_\delta)$ of the singular fiber $\L_0$, where $B_\delta := \{ a\in \R^2~|~ |a|<\delta\}$, contains no other singular points and admits a Hamiltonian $S^1$-action that is free everywhere except for $n$ focus-focus points which remain fixed.          
        \item  The generator of this $S^1$-action is well defined in the whole neighborhood $U(\L_0)$  and, in a neighborhood of each focus-focus point, coincides (up to sign)  with the function $f_2$  from Proposition \ref{prop:norm}. In particular, the functions $f_2(H,F)$ related to different focus-focus points coincide as functions of $H$ and $F$ (up to sign).

        \item Each non-singular fiber $\L_a := \mathcal F^{-1}(a)$, $a \in B_\delta \setminus \{0\}$, is connected and diffeomorphic to a 2-torus. 
     \end{enumerate}
\end{theorem}

We refer to the singularity described in this theorem as a  \textit{(symplectic) focus-focus singularity of complexity $n$} (or \textit{$n$-pinched focus-focus singularity}), where $n$ is the number of focus-focus points on the singular fiber $\L_0$.


The following important classification result is also due to Matveev and Zung.
\begin{theorem} \label{homeo}
    All focus-focus singularities of the same complexity $n$ are fiberwise homeomorphic.
\end{theorem}
 The case of complexity one was studied much earlier by L.M. Lerman, and Ya.L. Umanskii, see  \cite{LermUm} and references therein. 

\medskip


Since we are going to study smooth invariants of focus-singularities,  the main part of our construction will not use explicitly any symplectic structure.  For this reason, it is natural to define focus-focus singularities as in Theorem \ref{th:focus} but in a more general context without referring to a symplectic structure. In Section \ref{sec:sympl} we will show that all such singularities can be ``symplectized''.
 
 \begin{definition}\label{def:smoothffpoint}Consider a smooth map $\mathcal F: M^4 \to N^2$ and let $P\in M^4$ be a singular point of $\mathcal F$ with $\mathcal F(P)=Q$.  We will say that $P$ is of \textit{focus-focus type (in the smooth sense)}  if locally in some suitable complex coordinate  systems ($u, v$ in a neighborhood of $P$ on $M^4$ and $z$ in a neighborhood of $Q \in \N^2$)  the map $\mathcal F$ is given as $z=uv$.  
 \end{definition}

To define a focus-focus singularity in the semilocal setting,  we consider the whole fiber $\L_Q=\mathcal F^{-1}(Q)$ containing several critical points.  First of all, we impose the following two natural assumptions:

\begin{assumption} The singular fiber  $\L_Q$ is compact. \end{assumption}

\begin{assumption} All singular points of $\mathcal F$ located on $\L_Q$ are of focus-focus type. (In particular, the number of such points is finite.) 
\end{assumption}

These two assumptions, however, are not enough to determine the topology of the singularity.   Indeed, while in a neighborhood of each focus-focus point the structure of the singularity is standard and is described by Proposition~\ref{prop:local}, these local standard singularities can be arranged together in many different ways. So we need to assume that this arrangement is the same as in Figure~\ref{fig1}:
\begin{assumption}\label{assumption3} The singular fiber $\L_Q$ is homeomorphic to the $n$-pinched torus shown in Figure~\ref{fig1}.  In other words, the complement of focus-focus points in the singular fiber $\L_Q$  is a disjoint union of $n$ cylinders.
\end{assumption}
Finally, we need the following ``orientability'' assumption:
\begin{assumption}\label{assumption4} The manifold $M^4$ is oriented and the following equivalent conditions hold. \begin{enumerate} \item All intersections between $2$-spheres constituting the singular fiber $\L_Q$ are positive\footnote{Note that these spheres can be simultaneously oriented by picking a (local) orientation of the base $N^2$. Reversal of orientation of $N^2$ changes orientations of all spheres, so the intersection numbers are well-defined.}.
\item For any focus-focus points  $P \in \L_Q$, the orientation induced on $M^4$ by the local volume form $\mathrm{Re}\, (du\wedge dv) \wedge \mathrm{Re}\, (du\wedge dv)$, where $u, v$ are normal coordinates near $P$ (see Definition~\ref{def:smoothffpoint}), is positive. 
\end{enumerate}
\end{assumption}
 
 \begin{remark}
These two conditions are equivalent because the  fiber $\L_Q$ is given, in local normal coordinates, by two discs $\{u= 0\} $, $\{v = 0\}$ whose intersection is positive. Also note that this positivity condition imposes strong restrictions on the intersection form of $M^4$, see the paper \cite{GSmirnov} devoted to topology of multi-pinched focus-focus singularites.
 \end{remark}

\begin{definition}\label{def:focussmooth}
A  singularity satisfying Assumptions 1-4  will be called a {\it focus-focus} singularity {\it in the smooth sense}, or a \textit{smooth focus-focus singularity}.

\end{definition}

\begin{remark}
Note that in integrable nonholonomic systems Assumption \ref{assumption4} does not need to hold.  This phenomenon can be  understood as a topological obstruction to Hamiltonization of such systems, see \cite{Cushman}. Nevertheless, we believe that our classification is still valid for singularities not satisfying Assumption \ref{assumption4}. In this case, one regards the ``signs'' of focus-focus points (i.e., the signs of intersections  between $2$-spheres constituting the singular fiber) as additional discrete invariants.
\end{remark}

\section{Smooth structures on focus-focus singularities}\label{sec:smooth}
In what follows, given any smooth manifolds $M$ and $N$ and any point $P \in M$,  the notation $\Cont^\infty_P(M, N)$  stands for the space of germs at $P$ of smooth maps from $M$ to $N$, while $\Diff_P(M)$ is the group (under composition) of germs at $P$ of local diffeomorphisms of $M$ fixing $P$. Even if the manifolds $M$ and $N$ are complex, we assume that all maps are only infinitely real-differentiable, but not necessarily holomorphic. We also write $\F \colon (M,L) \to (N,Q)$ when $\F$ is a germ at  a submanifold $L \subset M$ of a map $M \to N$ taking $L$ to the point $Q \in N$.

\subsection{The group of liftable diffeomorphisms}
In this section we define the group of so-called \textit{liftable diffeomorphisms}, which, roughly speaking, determine possible way to ``shuffle'' the fibers near a focus-focus singular point. The sturcture of this group (in particular the fact that not all diffeomorphisms are liftable) underlies our construction of smooth invariants.

Let $\F \colon M^4 \to N^2$ be a map from a $4$-manifold $M^4$ to a surface $N^2$. Assume that $\F$ has a focus-focus singular point at $P \in M^4$. According to Definition \ref{def:smoothffpoint}, this means that there exist complex coordinates $(u,v)$ centered at $P$ and a complex coordinate $z$ centered at $Q := \F(P)$ such that in these coordinates the map $\F$ takes the form $z = uv$. In other words, there exist germs of diffeomorphisms $\Phi \colon (M^4, P) \to (\Complex^2, 0)$ and $\phi \colon (N^2, Q) \to (\Complex, 0)$ such that the following diagram commutes
\begin{equation}\label{comDiag}
\centering
\begin{tikzcd}
  (M^4, P) \arrow{r}[]{\Phi}  \arrow{d}{\F}&  (\Complex^2, 0)    \arrow{d}[]{uv}\\
(N^2, Q) \arrow{r}{\phi}& (\Complex, 0).
\end{tikzcd}
\end{equation}
We refer to $\phi$ as a \textit{normal chart}.
Such a { chart} is not unique. The collection of all normal charts is an intrinsic property of a focus-focus singular point. We describe this collection by means of so-called \textit{liftable germs}. Given two normal charts $\phi,  \tilde \phi \colon (N^2, Q) \to (\Complex, 0)$, we get the following diagram:
\begin{equation}\label{diag2}
\centering
\begin{tikzcd}
 (\Complex^2, 0)  \arrow[bend left, dashed, <-]{rr}[]{\tilde \Phi\, \circ\, \Phi^{-1}}    \arrow{d}[]{uv}&  (M^4, P) \arrow{r}[]{\Phi}  \arrow{l}[swap]{\tilde \Phi}  \arrow{d}{\F}&  (\Complex^2, 0)    \arrow{d}[]{uv}\\
 (\Complex, 0)  \arrow[bend right, dashed, <-]{rr}[swap]{\tilde \phi\, \circ\, \phi^{-1}}& (N^2, Q) \arrow{r}{\phi} \arrow{l}[swap]{\tilde \phi}& (\Complex, 0).
\end{tikzcd}
\end{equation}
From this diagram we conclude that the transition map $\tilde \phi \circ  \phi^{-1}  \in \Diff_0(\Complex)$ between the normal charts $\phi$ and $\tilde \phi$ admits a lift $\tilde \Phi \circ \Phi^{-1}$ to the total space of the fibration $uv \colon (\Complex^2, 0) \to (\Complex, 0)$. We call such germs \textit{liftable}: \begin{definition}\label{def:liftable}
A germ of a diffeomorphism $\psi \in \Diff_0(\Complex)$ is called \textit{liftable} if there exists a germ of a diffeomorphism $\Psi \in \Diff_0(\Complex^2)$ such that the following diagram commutes:
\begin{equation}
\centering
\begin{tikzcd}\label{diag3}
 (\Complex^2, 0)  \arrow{r}[]{\Psi}  \arrow{d}{uv}&  (\Complex^2, 0)    \arrow{d}[]{uv}\\
 (\Complex, 0)   \arrow{r}{\psi} & (\Complex, 0).
\end{tikzcd}
\end{equation}
\end{definition}
liftable germs form a subgroup of the group $\Diff_0(\Complex)$. We denote this subgroup by ${\ADiff_0(\Complex)}$.  From diagram~\eqref{diag2} we get the following result.
\begin{proposition}\label{charOfliftable}
Let $\phi \colon (N^2, Q) \to (\Complex, 0)$ be a normal chart, and let $\psi \in \ADiff_0(\Complex)$ be a liftable germ. Then $\psi\circ \phi   \colon (N^2, Q) \to (\Complex, 0)$ is also a normal chart. Conversely, for any two normal charts $\phi, \tilde \phi \colon (N^2, Q) \to (\Complex, 0)$  the corresponding transition map $\tilde \phi \circ \phi^{-1} \in \Diff_0(\Complex)$ is liftable.\par
In other words, the collection of normal charts $(N^2, Q) \to (\Complex, 0)$ is a principal homogeneous space relative to the left action of $\ADiff_0(\Complex)$.
\end{proposition}
The following result classifies liftable germs.
\begin{theorem}\label{liftable}
    A germ of a diffeomorphism $\psi \in \Diff_0(\Complex)$ is liftable if and only if it can be written either as
        $$\psi (z)=z  h(z),$$or as 
                 $$\psi (z)=\bar{z}   h(z),$$ where, in both cases, $h \in  \Cont^\infty_0(\Complex , \Complex)$ is a germ at $0$ of an  infinitely real-differentiable function $\Complex \to \Complex$  with $h(0) \neq 0$. 
  \end{theorem}
  \begin{remark}\label{rem:adm}Note that an infinitely real-differentiable function $\psi \colon \Complex \to \Complex$ is divisible by $z$ (or $\bar z$)  if and only if {its Taylor series} at $0$ in terms of $z, \bar z$ is divisible by $z$ (respectively, $\bar z$).
    So,  Theorem \ref{liftable} can be reformulated as follows: $\psi \in \Diff_0(\Complex)$ is liftable if and only if {its Taylor series} at $0$ is either divisible by $z$, or divisible by $\bar z$. (Equivalently, the Taylor series of $\psi$ either does not contain monomials of the form $\bar z^k$, or does not contain monomials of the form $ z^k$.)

  \end{remark}
\begin{proof}[Proof of Theorem \ref{liftable}]
First note that the complex conjugation map $\psi \colon z \to \bar z$ is liftable: as its lift $\Psi$ making diagram \eqref{diag3} commute, one can take $\Psi(u,v) := (\bar u, \bar v)$. So, it suffices to show that an orientation-preserving diffeomorphism $\psi \in \Diff_0(\Complex)$ is liftable if and only if it can be written as $\psi (z)=z  h(z)$.\par

Assume that $\psi(z) = zh(z)$. Then diagram \eqref{diag3} commutes, for instance, for  $\Psi(u,v) := (u, vh(uv))$, so $\psi$ is liftable.\par
  Conversely, assume that $\psi$ is orientation-preserving and liftable. According to Remark~\ref{rem:adm}, it suffices to show that the Taylor series $\psi_\infty  \in \Complex[[z ,\bar z]]$ of $\psi$ at $0$ is divisible by $z$. Let $\Psi$ be a lift of $\psi$ making diagram \eqref{diag3} commute. Then the Taylor series $f_\infty,g_\infty \in \Complex[[u,\bar u, v, \bar v]]$ of components of $\Psi$ satisfy
 \begin{equation}\label{prod}
  \psi_\infty(uv, \bar u\bar v) = f_\infty(u,\bar u, v, \bar v)\cdot g_\infty(u,\bar u, v, \bar v).
\end{equation}
Equating the lowest-degree terms on both sides, we get that the quadratic part $a uv + b \bar u\bar v$ of $ \psi_\infty(uv,\bar u\bar v) $ is the product of linear parts of $f_\infty$ and $g_\infty$. But since a quadratic form  can only be factored into linear forms when its rank is at most $ 2$, it follows that either $a = 0$ or $b = 0$. Also taking into account that $\psi$ is an orientation-preserving diffeomorphism, we conclude that $a \neq 0$ and $b = 0$, i.e. the quadratic part of $ \psi_\infty(uv,\bar u\bar v) $ is a non-zero multiple of $uv$. But then it follows that the linear part of $g_\infty$ is a non-zero multiple of either $u$ or $v$. Without loss of generality we can assume that it is $v$, i.e. $g_\infty = cv + \dots$, where $c$ is a non-zero constant. Then, setting $u$ in equation \eqref{prod} to zero, we get
 \begin{equation}\label{prod2}
  \psi_\infty(0, \bar u\bar v) = f_\infty(0,\bar u, v, \bar v)\cdot ( cv +\dots).
\end{equation}
Assume that the left-hand side does not vanish. Then, equating the lowest-degree terms on both sides, we get that a certain non-zero polynomial of $\bar u, \bar v$ is divisible by $v$. Since this is not possible, it follows that $ \psi_\infty(0, \bar u\bar v) $ is identically zero, and hence $ \psi_\infty(z, \bar z)$ is divisible by $z$, as desired.
\end{proof}
\subsection{Gluing maps and the main classification theorem}\label{sec:gluingMaps}
In this section we apply the above description of liftable diffeomorphisms to give an algebraic description for the space of $n$-pinched focus-focus singularities up to smooth equivalence.\par
Assume that the fiber of $\F \colon M^4 \to N^2$ over $Q \in N^2$ contains $n$ focus-focus points $P_1, \dots, P_n$ (in the smooth sense). For each of those points, choose a normal chart $\phi_i \colon (N^2, Q) \to (\Complex, 0)$.
\begin{definition}
The germs of diffeomorphisms $\phi_{i,j} \in  \Diff_0(\Complex)$ given by \begin{align}\label{gluingMapFormula}\phi_{i,j} := \phi_i \circ \phi_j^{-1}\end{align} are called \textit{gluing maps} of a focus-focus singularity (relative to the normal charts $\phi_1, \dots, \phi_n$).
\end{definition}
\begin{remark}\label{rem:cocycle}
Gluing maps satisfy the conditions
\begin{equation}\label{cocycle}
\phi_{i,i} = \id, \quad \phi_{i,j}\circ\phi_{j,k}\circ\phi_{k,i}= \id
\end{equation}
 and hence are uniquely determined by the choice of, say, $\phi_{1,2}, \phi_{1,3}, \dots, \phi_{1, n}$. Note that while $\phi_{1,2}, \phi_{2,3}, \dots, \phi_{n-1,n}$ seems to be a more natural choice, it is less convenient for computations.
 \end{remark}
 Possible non-triviality of gluing maps is the source of smooth invariants for focus-focus singularities. Note that although gluing maps depend on the choice of normal charts, they are well-defined up to the left-right action of liftable diffeomorphisms. Hence one can identify smooth structures on focus-focus singularities with the corresponding quotient space.
 \begin{remark}
Here and in what follows, we assume that singular points of a focus-focus singularity are labeled, in cyclic order, with integers $\{1, \dots, n\}$, and all diffeomorphisms are required to preserve this labelling. Hence, all invariants we construct are invariants of ``labeled focus-focus singularities''. To obtain invariants of unlabelled singularities, one should take into account the action of the ``relabelling'' group, isomorphic to the dihedral group $ D_n$.
\end{remark}
\begin{theorem}\label{thm:thm1}
\begin{enumerate} \item
Two smooth focus-focus singularities with the same number of singular points are diffeomorphic if and only if the corresponding gluing maps are related by
 \begin{equation}\label{transformationRule}
 \tilde \phi_{i,j} = \psi_i\circ \phi_{i,j} \circ \psi_j^{-1},
\end{equation}
where $\psi_1, \dots, \psi_n \in \ADiff_0(\Complex)$ are liftable.
\item  Smooth structures on an $n$-pinched focus-focus singularity are in one-to-one correspondence with orbits of the $\ADiff_0(\Complex)^n$ action on $\Diff_0(\Complex)^{n-1}$ given by
 \begin{equation}\label{gaugeAction}
(\psi_1, \dots, \psi_n) \acts (\phi_{1,2},  \dots, \phi_{1,n}) :=  (\psi_1\circ\phi_{1,2}\circ\psi_2^{-1},  \dots, \psi_{1}\circ\phi_{1, n}\circ\psi_{n}^{-1}),
\end{equation}
where 
$\psi_1, \dots, \psi_n\in \ADiff_0(\Complex)$
and 
$\phi_{1,2},  \dots, \phi_{1,n} \in \Diff_0(\Complex)$. 
\item Equivalently, smooth structures on an $n$-pinched focus-focus singularity are in one-to-one correspondence with orbits of the $\ADiff^\infty_0(\Complex)^n$ action on $\Diff^\infty_0(\Complex)^{n-1}$ given by the same formula \eqref{gaugeAction}. Here $\Diff^\infty_0(\Complex)$ is the group of $\infty$-jets at $0$ of local diffeomorphisms $(\Complex, 0) \to (\Complex, 0)$, and $\ADiff_0^\infty(\Complex)$ is the subgroup of liftable $\infty$-jets, i.e., series divisible by $z$ or $\bar z$.
\end{enumerate}
\end{theorem}
\begin{remark}
The last statement of the theorem can be interpreted as follows: Two smooth focus-focus singularities are $\Cont^\infty$-equivalent if and only if they are formally equivalent.
\end{remark}
The proof of Theorem \ref{thm:thm1} is based on the following lemma.
\begin{lemma}\label{gluingMapsDetermineTheSingulatiry}
Assume that we are given two $n$-pinched focus-focus singularities $\F \colon (M^4, L) \to (N^2, Q)$ and $\tilde \F \colon (\tilde M^4, \tilde L) \to (\tilde N^2, \tilde Q)$ such that for suitable choice of normal charts the corresponding gluing maps coincide (i.e. if  $\phi_{i,j}$'s are gluing maps for $\F$, and $\tilde \phi_{i,j}$'s are gluing maps for $\tilde \F$, then $\phi_{i,j} = \tilde \phi_{i,j}$). Then these singularities are diffeomorphic.
\end{lemma}
\begin{proof}[Proof of the lemma]
We need to show that there exist germs of diffeomorphisms $\psi \colon (N^2, Q) \to  (\tilde N^2, \tilde Q)$ and $\Psi \colon (M^4, L) \to (\tilde M^4, \tilde L)$ such that the following diagram commutes:
\begin{equation}\label{diag4}
\centering
\begin{tikzcd}
  (M^4, L) \arrow{r}[]{\Psi}  \arrow{d}{ \F} &   (\tilde M^4, \tilde L)  \arrow{d}{\tilde \F}\\
(N^2, Q)   \arrow{r}{\psi}&  (\tilde N^2, \tilde Q).
\end{tikzcd}
\end{equation}

 To begin with, we construct these maps $\psi$ and $\Psi$ locally. Let $P_1, \dots, P_n$ be the singular points of $\F$ on the fiber $L$, and let $\tilde P_1, \dots, \tilde P_n$ be the singular points of $\tilde \F$ on the fiber $\tilde L$. Let also $\phi_i \colon (N^2, Q) \to (\Complex, 0) $ and  $\tilde \phi_i \colon (\tilde N^2, \tilde Q) \to (\Complex, 0) $ be normal charts satisfying the condition of the lemma. 
 Then, combining diagrams~\eqref{comDiag} for normal charts $\phi_i$ and $\tilde \phi_i$, we get the following commutative diagram
\begin{equation}\label{diag5}
\centering
\begin{tikzcd}
  (M^4, P_i) \arrow{r}[]{\Phi_i}  \arrow{d}{\F}  \arrow[bend left, dashed, ]{rr}[]{ \tilde \Phi_i^{-1} \circ \,\Phi_i}  &  (\Complex^2, 0)    \arrow{d}[]{uv} &   (\tilde M^4, \tilde P_i) \arrow{l}[swap]{\tilde \Phi_i}  \arrow{d}{\tilde \F}\\
(N^2, Q)  \arrow[bend right, dashed, ]{rr}[swap]{ \tilde \phi_i^{-1} \circ\, \phi_i}  \arrow{r}{\phi_i}& (\Complex, 0) & (\tilde N^2, \tilde Q) \arrow{l}[swap]{\tilde \phi_i}.
\end{tikzcd}
\end{equation}
This diagram can be viewed as a local version of \eqref{diag4}, with $\psi = \tilde \phi_i^{-1}\circ \phi_i $ and $\Psi = \tilde \Phi_i^{-1} \circ \Phi_i$.
Furthermore, the coincidence of the gluing maps $\phi_i \circ \phi_j^{-1} = \tilde \phi_i  \circ\tilde \phi_j^{-1} $ implies 
$$\tilde \phi_i^{-1}\circ \phi_i  = \tilde \phi_j^{-1}\circ \phi_j,
$$ i.e.  $\psi$ does not depend on $i$. (In other words, the lower dashed arrow in all $n$ copies of diagram~\eqref{diag5} is the same.) 
Now it remains to extend local diffeomorphisms $\Psi = \tilde \Phi_i^{-1} \circ \Phi_i \colon   (M^4, P_i)  \to (\tilde M^4, \tilde P_i) $ to a global one $ \Psi \colon (M^4, L) \to (\tilde M^4, \tilde L)$ and hence obtain a global version of \eqref{diag4}. To that end, notice that by Assumption~\ref{assumption3} of Definition~\ref{def:focussmooth}, the neighborhood of an $n$-pinched focus-focus singular fiber can be represented as a union of $n$ standard neighborhoods of focus-focus points and $n$ trivial fibrations into cylinders. The local diffeomorphisms 
$\tilde \Phi_i^{-1} \circ \Phi_i $ define maps between standard neighborhoods of focus-focus points and hence between boundaries of the cylinders. (More precisely, they define $\infty$-jets of diffeomorphisms of cylinders relative to the boundary. Also note that one may need to compose $\Phi_i$ with the map $(u,v) \mapsto (v,u)$ to ensure that opposite boundaries of each cylinder are mapped to opposite boundaries of another cylinder.) These diffeomorphisms between the boundaries can be extended inside the cylinders thanks to Assumption~\ref{assumption4} of Definition~\ref{def:focussmooth}. (This extension can be constructed as follows. First, we identify the cylinders using an arbitrary diffeomorphism. This reduces the problem to the following: Given a cylinder and  $\infty$-jets of orientation-preserving diffeomorphisms at its boundaries, one needs to find a global diffeomorphism of the cylinder to itself which realizes given jets. Such an extension can be perfomed, for instance, using the Moser path method. First, one extends given jets to actual diffeomorphisms in small neighborhoods of the boundary. Then one connects those diffeomorphisms with the identity. These paths of diffeomorphisms can be viewed as flows of certain time-dependent vector fields near the boundary. Using a partition of unity, one extends those vector fields to a global vector field on the cylinder. Integrating the latter vector field provides the desired diffeomorphism.) This extension gives us a global diffeomorphism $\Psi\colon (M^4, L) \to (\tilde M^4, \tilde L)$ making diagram \eqref{diag4} commute. Thus, the lemma is proved.
\end{proof}

\begin{proof}[Proof of Theorem \ref{thm:thm1}]
We begin with the first statement.
Assume that two focus-focus singularities $\F \colon (M^4, L) \to (N^2, Q)$ and $\tilde \F \colon (\tilde M^4, \tilde L) \to (\tilde N^2, \tilde Q)$  are diffeomorphic.
This means that there exist germs of diffeomorphisms $\psi \colon (N^2, Q) \to  (\tilde N^2, \tilde Q)$ and $\Psi \colon (M^4, L) \to (\tilde M^4, \tilde L)$ such that diagram \eqref{diag4} commutes. Take any normal charts $\phi_1, \dots, \phi_n \colon (N^2, Q)   \to (\Complex, 0)$ for the first singularity, and ``push them forward'' using the bottom arrow $\psi$ in diagram~\eqref{diag4}, i.e. consider the charts $ \phi_i \circ \psi^{-1} \colon  (\tilde N^2, \tilde Q)   \to (\Complex, 0)$ on the base of the second singularity. It is easy to see that these charts are normal. (Diffeomorphisms of focus-focus singularities preserve normality.) Therefore, if $\tilde \phi_1, \dots, \tilde \phi_n$ are any other normal charts for the second singularity, then by  Proposition~\ref{charOfliftable} we have 
 $$
 \tilde \phi_i = \psi_i \circ  \phi_i \circ \psi^{-1},
 $$
 where $\psi_i$ is liftable. But this immediately yields relation \eqref{transformationRule} between the gluing maps $\phi_{i,j} = \phi_i \circ \phi_j^{-1}$ and $\tilde \phi_{i,j} = \tilde \phi_i \circ \tilde \phi_j^{-1}$.

Conversely, assume that we are given two singularities such that, for certain normal charts $\phi_1, \dots, \phi_n$ for the first singularity and $\tilde \phi_1, \dots, \tilde \phi_n$  for the second one, the corresponding gluing maps are related by~\eqref{transformationRule}. 
Then, since $\phi_i$ is a normal chart, the chart  $\psi_i \circ \phi_i$, where $\psi_i$ is a liftable diffeomorphism entering~\eqref{transformationRule}, is normal as well  (Proposition \ref{charOfliftable}). Furthermore, gluing maps for the normal charts $ \psi_1 \circ \phi_1, \dots,  \psi_n \circ \phi_n$ are the same as for $\tilde \phi_1, \dots, \tilde \phi_n$:
$$
  ( \psi_i \circ \phi_i ) \circ ( \psi_j\circ \phi_j )^{-1} =  \psi_i \circ \phi_i \circ \phi_j^{-1} \circ \psi_j^{-1} = \tilde \phi_i \circ \tilde  \phi_j^{-1}.
$$
So, the singularities are diffeomorphic by Lemma \ref{gluingMapsDetermineTheSingulatiry}. 
Thus, the first statement of the theorem is proved.

To prove the second statement, we use that any collection of diffeomorphisms $\phi_{i,j} \in \Diff_0(\Complex)$ satisfying \eqref{cocycle} can be realized as gluing maps for an appropriate focus-focus singularity. Such a singularity can be obtained by taking standard neighborhoods of focus-focus points and identifying neighborhoods of their boundaries (which are trivial foliations into cylinders) as prescribed by the maps $\phi_{i, i+1}$. (The orientations of the boundaries should be matched properly for the resulting singularity to satisfy Assumption \ref{assumption4}. Note that since we glue neighborhoods of boundaries, the resulting space automatically obtains a smooth structure.) Therefore, a smooth structure on a focus-focus singularity is determined by a collection $\{\phi_{i,j} \in \Diff_0(\Complex)\}$  satisfying~\eqref{cocycle} modulo the action defined by~\eqref{transformationRule}. But since such a collection $\{\phi_{i,j}\}$
is uniquely determined by $\phi_{1,2}, \dots, \phi_{1,n}$, this reduces to the action of  $\ADiff_0(\Complex)^n$ on $\Diff_0(\Complex)^{n-1}$ given by \eqref{gaugeAction}.

%
To prove the last statement, consider the map
$ \Diff_0(\Complex)^{n-1}  \to \Diff^\infty_0(\Complex)^{n-1}$
which takes a collection of germs to the corresponding jets. This map is surjective by Borel's theorem on the existence of a smooth map with a given Taylor series. Furthermore, this map intertwines $\ADiff_0(\Complex)^n$ action on $ \Diff_0(\Complex)^{n-1}$ with the $\ADiff_0^\infty(\Complex)^n$ action on $ \Diff_0^\infty(\Complex)^{n-1}$, which gives a surjective map between the corresponding orbit spaces
 $$ \Diff_0(\Complex)^{n-1} \, / \, \ADiff_0(\Complex)^n \to \Diff^\infty_0(\Complex)^{n-1} \, / \, \ADiff_0^\infty(\Complex)^n. $$
To complete the proof it suffices to notice that since flat diffeomorphisms are liftable (see Remark \ref{rem:adm}), the latter map is also injective and hence a bijection. Thus, the theorem is proved.
%
\end{proof}
 \begin{remark}\label{rem:infinitejets}
 In what follows, we prefer to work with orientation-preserving gluing maps. Let $\Diff_0(\Complex)_+ \subset \Diff_0(\Complex)$ be the subgroup of orientation-preserving germs. Then, since complex conjugation is liftable, each orbit of the action $\ADiff_0(\Complex)^n \acts \Diff_0(\Complex)^{n-1}$ has a representative which belongs to $\Diff_0(\Complex)^{n-1}_+$. Also note than the action of an element $(\psi_1, \dots, \psi_n) \in \ADiff_0(\Complex)^n$ on $ \Diff_0(\Complex)^{n-1}$ preserves  $\Diff_0(\Complex)_+^{n-1}$ if and only if either all $\psi_i$'s are orientation-preserving, or all of them are orientation-reversing. Denote by $\ADiff_0(\Complex)_+ \subset \ADiff_0(\Complex)$ the subgroup of orientation-preserving liftable germs, and let $\ADiff_0(\Complex)_- := \ADiff_0(\Complex) \, \setminus \, \ADiff_0(\Complex)_+$ be orientation-reversing liftable germs. Then we get a natural identification between orbits spaces
 $$ \Diff_0(\Complex)^{n-1} \, / \, \ADiff_0(\Complex)^n \simeq  \Diff_0(\Complex)^{n-1}_+ \, / \, \ADiff_0(\Complex)_\pm^n, $$
 where $\ADiff_0(\Complex)_\pm^n:= \ADiff_0(\Complex)_+^n \sqcup \ADiff_0(\Complex)_-^n$, and the action of $\ADiff_0(\Complex)_\pm^n$ on $\Diff_0(\Complex)^{n-1}_+$ is defined by the same formula~\eqref{gaugeAction}.  
 \end{remark}
  \begin{corollary}\label{cor2.11}
 Smooth structures on an $n$-pinched focus-focus singularity are in one-to-one correspondence with orbits of the $\ADiff_0(\Complex)_\pm^n$ {action} on $\Diff_0(\Complex)^{n-1}_+$ defined by~\eqref{gaugeAction}. 
 \end{corollary}
 
\begin{remark}
In the complex-analytic setting, focus-focus fibers are known as Kodaira $I_n$ singularities of elliptic fibrations. In this case all gluing maps are holomorphic and hence liftable. Therefore, by Theorem \ref{thm:thm1} any such singularity is diffeomorphic to the one with trivial gluing maps. In fact, a stronger statement is true: any two $I_n$ singularities with the same $n$ are complex isomorphic.
\end{remark} 
 
 \subsection{Description of first order invariants}
 Since the groups $\ADiff_0(\Complex)$ and $\Diff_0(\Complex)$ are infinite-dimensional, an explicit description of orbits for action \eqref{gaugeAction} is a problem of unknown complexity. (See, however, the description of generic orbits for $n=2$ in Section \ref{sec:dpinch}). Nevertheless, one can construct invariants of this action by replacing the groups $\ADiff_0(\Complex)$ and $\Diff_0(\Complex)$ with the corresponding finite-dimensional groups of finite-order jets. The aim of this section is to define these invariants and explicitly describe those which are related to $1$-jets.

 Let $\Diff^k_0(\Complex)$ be the group of $k$-jets at $0$ of diffeomorphisms $(\Complex, 0) \to (\Complex, 0)$, and $\ADiff_0^k(\Complex)$ be the subgroup of liftable $k$-jets, i.e., jets divisible by $z$ or $\bar z$. (Note that both $\Diff^k_0(\Complex)$ and $\ADiff_0^k(\Complex)$ are finite-dimensional real Lie groups). Let also  $\Diff_0^k(\Complex)_+ \subset \Diff_0^k(\Complex)$ be the subgroup of orientation-preserving jets, and let 
 $ \ADiff_0^k(\Complex)_+ :=  \Diff_0^k(\Complex)_+ \cap \ADiff_0^k(\Complex)$, $\ADiff_0^k(\Complex)_- := \ADiff_0^k(\Complex) \, \setminus \, \ADiff_0^k(\Complex)_+$ (cf. Remark \ref{rem:infinitejets}).
 
  Then we have surjective homomorphisms $ \Diff_0(\Complex)_+ \to \Diff^k_0(\Complex)_+$ assigning to each germ the corresponding $k$-jet. These homomorphisms induce surjective maps between orbit spaces
 \begin{equation}\label{mapsBetweenOrbitSpaces}
\Diff_0(\Complex)^{n-1}_+ \, / \, \ADiff_0(\Complex)_\pm^n \to \Diff_0^k(\Complex)^{n-1}_+ \, / \, \ADiff_0^k(\Complex)_\pm^n,
\end{equation}
 where $\ADiff_0^k(\Complex)_\pm^n:= \ADiff_0^k(\Complex)_+^n \sqcup \ADiff_0^k(\Complex)_-^n$, and the action of $\ADiff_0^k(\Complex)_\pm^n$ on $\Diff^k_0(\Complex)^{n-1}_+$ is given by the same formula \eqref{gaugeAction}. In other words, invariants of the action  $\ADiff_0^k(\Complex)_\pm^n \acts \Diff^k_0(\Complex)^{n-1}_+$  are also invariants of the action  $\ADiff_0(\Complex)_\pm^n \acts\Diff_0(\Complex)^{n-1}_+$, i.e. invariants of $n$-pinched focus-focus singularities. We say that such invariants have \textit{order k}. In this section we describe the first order invariants.\par
 The group $\Diff^1_0(\Complex)$ is isomorphic to $\GL_2(\R)$, but it will be convenient to regard its elements as invertible $\R$-linear functions from $\Complex$ to $\Complex$, that are functions of the form $a z + b \bar z$, where $a,b \in \Complex$ are such that $|a| \neq |b|$. The subgroup $\Diff^1_0(\Complex)_+$ consists of orientation-preserving $\R$-linear functions from $\Complex$ to $\Complex$, that are functions of the form $a z + b \bar z$, where $a,b \in \Complex$ are such that $|a| > |b|$. The subgroup $\ADiff^1_0(\Complex)$ consists of invertible $\Complex$-linear and $\Complex$-antilinear functions $\Complex \to\Complex$, that are functions of the form $az$, where $a \in \Complex^*$ (such functions constitute the subgroup $\ADiff^1_0(\Complex)_+$), or $b\bar z$, where $b \in \Complex^*$ (such functions form the complimentary subset $\ADiff^1_0(\Complex)_-$).\par
 \begin{proposition}\label{firstOrderInavriants}
 \begin{enumerate}\item
 Every orbit of the action  $\ADiff_0^1(\Complex)^n_\pm\acts\Diff^1_0(\Complex)^{n-1}_+$ has a representative of the form
\begin{equation}\label{reprForm}
 (z + \mu_{1}\bar z, \dots, z+ \mu_{n-1} \bar z), 
\end{equation}
where $\mu_i \in \Complex$, $|\mu_i| < 1$ for each~$i$.
 This element is unique up to multiplying all $\mu_i$'s by the same complex number of absolute value $1$ and replacing each $\mu_i$ by $\bar \mu_i$ (or performing both operations at a time).  Thus, the orbits are parametrized by $n-1$ numbers $\mu_1, \dots, \mu_{n-1}$ in the open unit disk $\{ z \in \Complex \mid |z| < 1\}$, considered up to multiplication by the same complex number of absolute value $1$ and simultaneous complex conjugation.


\item The numbers $\mu_i$ corresponding to the orbit of  $(a_1 z + b_1 \bar z, \dots, a_{n-1} z + b_{n-1} \bar z ) \in \Diff^1_0(\Complex)^{n-1}_+$ are given by
$$
\mu_i = \frac{b_i}{\bar a_i}.
$$

\end{enumerate}
 \end{proposition}
 \begin{proof}
Take $\xi := (a_1 z + b_1 \bar z, \dots, a_{n-1}z + b_{n-1}\bar z ) \in \Diff^1_0(\Complex)^{n-1}_+$. 
Acting by $\eta := (z, a_1z, \dots, a_{n-1}z) \in \ADiff_0^1(\Complex)^n_+$ on $\xi$, we get
$$
 (z \circ (a_1 z + b_1 \bar z) \circ (a_1^{-1}z), \dots ) =  (z + \mu_{1}\bar z, \dots),
$$
 where $\mu_i = {b_i}{\bar a_i}^{-1}$ (see formula \eqref{gaugeAction} for the action). This proves the existence part of the first statement, as well as the second statement. (Note that $|b_i  \bar a_i^{-1}| < 1$ thanks to the orientation-preserving condition.)\par
To prove the uniqueness part of the first statement, one checks that  an element of $ \ADiff_0^1(\Complex)^n_+$ of form  \eqref{reprForm}  is mapped,  under the action of $\eta \in \ADiff_0^1(\Complex)^n_\pm$ ,  to an element of the same form  if and only if $\eta =  (cz, \dots, cz) $ or $\eta = (c\bar z, \dots, c\bar z)$, where $c \in \Complex^*$. In the former case, the numbers $\mu_i$ in   \eqref{reprForm} are transformed by the rule $  \mu_i \mapsto {c}{\bar c}^{-1} \mu_i $, while in the latter case we get  $ \mu_i \mapsto {c}{\bar c}^{-1}  \bar \mu_i $. But since ${c}{\bar c}^{-1} $ can take any value on the unit circle, the result follows.
 \end{proof}
 \begin{corollary}\label{cor2.13}
 The orbit space $ \Diff^1_0(\Complex)^{n-1} \, / \, \ADiff_0^1(\Complex)^n$ is homeomorphic to the quotient of a polydisk
 $
 \{ z \in \Complex \mid |z| < 1\}^{n-1} $
by the diagonal action of the orthogonal group~$\mathrm{O}_1(\R)$. In particular, $\dim \left(  \Diff^1_0(\Complex)^{n-1} \, / \, \ADiff_0^1(\Complex)^n \right) = 2n-3$.
 \end{corollary}
%
Hence we get the following result.
 \begin{theorem}\label{theorem:firstOrderInv}
Focus-focus singularities with $n$ singular points have $2n-3$ first order invariants. These invariants are given by the numbers
$$
\mu_i := \frac{{\partial \phi_{1,i}}/{\partial\bar z}(0)}{\phantom{\int}\overline{{{\partial\phi_{1,i}}}/{\partial   z}}(0)\phantom{\int}}\in 
 \{ z \in \Complex \mid |z| < 1\}, \quad i= 2, \dots, n,
$$
considered up to multiplication by the same complex number of absolute value $1$ and simultaneous complex conjugation. Here $\phi_{1,2}, \phi_{1,3}, \dots \phi_{1,n} \in  \Diff_0(\Complex)_+$ are orientation-preserving gluing maps.
 \end{theorem}
 \subsection{First order invariants and complex structures}\label{sec:foics}
 In this section we give a geometric  interpretation of first order invariants. Later on, in Section \ref{sec:example}, we will generalize this construction to give an example of a singularity which does not admit a smooth almost direct product decomposition.\par

 Let $\F \colon (M^4, P) \to (N^2, Q) $ be a germ of a smooth map with a focus-focus singular point at $P$. Then every normal chart $\phi \colon (N^2, Q) \to (\Complex, 0)$ gives rise to a complex structure $J$ on the tangent space $\T_QN^2$, defined as the pullback of the canonical complex structure on $\Complex$ by means of $\phi$. In other words, we have the following commutative diagram
 \begin{equation}
 \begin{tikzcd}\label{diag6}
\T_QN^2  \arrow{r}[]{\diff_Q\phi}  \arrow{d}{J}& \T_0\Complex    \arrow{d}[]{J_{st}}\\
\T_QN^2  \arrow{r}{\diff_Q\phi} &  \T_0\Complex,
\end{tikzcd}
\end{equation}
 where $J_{st}$ is multiplication by $\i $, and $\diff_Q$ stands for the differential at $Q$.
 \begin{proposition}
 Complex structures on $\T_QN^2$ coming from different normal charts (associated with the same  focus-focus point $P$) agree up to sign.
 \end{proposition}
 \begin{proof}
 From diagram \eqref{diag6} we get
  $
J_i = (\diff_Q\phi_i)^{-1} \circ J_{st} \circ \diff_Q\phi_i,
$
while for a different normal chart $\tilde \phi$ we get 
\begin{gather}
\tilde J_i = (\diff_Q\tilde \phi_i)^{-1}  \circ J_{st} \circ  \diff_Q\tilde \phi_i \\ =  (\diff_Q\phi_i)^{-1}  \circ (\diff_Q\phi_i \circ (\diff_Q\tilde \phi_i)^{-1}) \circ J_{st} \circ (\diff_Q\phi_i \circ (\diff_Q\tilde \phi_i)^{-1})^{-1}\circ \diff_Q\phi_i \\
=  (\diff_Q\phi_i)^{-1}  \circ (\pm J_{st}) \circ  \diff_Q\phi_i = \pm J_i,
\end{gather}
  where we used that the germ $\phi  \circ \tilde \phi^{-1} \colon (\Complex, 0) \to (\Complex, 0)$ is liftable (see Proposition \ref{charOfliftable}) and hence its differential $\diff (\phi \circ \tilde \phi^{-1}) =  \diff\phi_i \circ \diff\tilde \phi_i^{-1}$ is complex or anti-complex  (Theorem~\ref{liftable}). 
  \end{proof}
So, we get a well-defined pair of complex structures $\pm J$ on  $\T_QN^2$.  Now assume that the fiber of $\F$ over $Q$ contains $n$ focus-focus points $P_1, \dots, P_n$. Then we get $n$ pairs of complex structures $\pm J_i$ on  $\T_QN^2$. By construction, these pairs, considered up to simultaneous conjugation, are invariant under diffeomorphisms. Since the space of complex structures on $\R^2$ is $2$-dimensional, while the conjugation action of $\GL_2(\R)$ has one-dimensional kernel consisting of scalar matrices, this way we get $2n-3$ smooth invariants of $n$-pinched fosus-focus singularities, cf. Theorem~\ref{theorem:firstOrderInv}.
 \begin{proposition}
 The invariants of the $n$-tuple $(\pm J_1, \dots, \pm J_n)$ under the $\GL_2(\R)$ action are exactly the first order invariants, as defined in Theorem~\ref{theorem:firstOrderInv}. In other words, two focus-focus singularities have conjugate tuples $(\pm J_1, \dots, \pm J_n)$ if and only if they have the same first order invariants.
 \end{proposition}
 \begin{proof}
 From diagram \eqref{diag6}, we have
$$
J_i = (\diff_Q\phi_i)^{-1}  \circ J_{st} \circ \diff_Q\phi_i,
$$
where $\phi_i \colon (N^2, Q) \to (\Complex, 0)$ is the normal chart corresponding to the $i$'th singular point. 
Since we are only interested in complex structures $(J_1, \dots, J_n)$ up to simultaneous conjugation, we may replace them by complex structures on $\Complex$ defined by
$$
\tilde J_i :=  \diff_Q\phi_1 \circ J_{i} \circ  (\diff_Q\phi_1)^{-1}.
$$
Then we have
\begin{align}\label{csgm}
\tilde J_1 = J_{st}, \quad \tilde J_i =\diff_0\phi_{1,i} \circ J_{st} \circ  (\diff_0\phi_{1,i})^{-1},
\end{align}
where $\phi_{1,i}$ are the gluing maps. Now it is easy to see that the $\ADiff_0^1(\Complex)^n$ action on differentials of the gluing maps corresponds to simultaneous conjugation of $\tilde J_i $'s and changing their signs. But this means that the invariants of $(\pm J_1, \dots, \pm J_n)$ are exactly the invariants of the $\ADiff_0^1(\Complex)^n$ action on $\Diff^1_0(\Complex)^{n-1}$, i.e. first order invariants.
 \end{proof}
  \begin{remark}
  For double pinched focus-focus singularities, the only first order invariant is $\mu = |\mu_2| \in [0,1)$ (see Theorem \ref{theorem:firstOrderInv} and Section \ref{sec:dpinch} below), while the only invariant of a pair $J_1, J_2$ of complex structures is the trace of $ J_2 J_1^{-1}$. (We get rid of the ambiguity in the choice of signs  by requiring that $J_1$ and $J_2$ define the same orientation on  $\T_QN^2$.) The relation between these invariants is as follows:
\begin{align}\label{trmu}
 \tr (J_2 J_1^{-1}) = 2\cdot\frac{1 + \mu^2}{1 - \mu^2}.
\end{align}
 There are also similar formulas for general $n$, with $|\mu_i|$ instead of $\mu$ in the right-hand side. However, for $n > 2$, the absolute values of $\mu_i$'s do not form a complete set of first order invariants (while the traces of ratios do not form a complete set of invariants for $n$-tuples of complex structures), so there are additional, more complicated, relations.
 \end{remark}
 

 In the remaining part of this section we explain how to construct complex structures on the tangent space to the base of a focus-focus fibration without referring to the classification of liftable diffeomorphisms. This geometric construction will be useful later on, in the discussion of the multidimensional case (see Section \ref{sec:example}).\par 
As above,  let $\F \colon (M^4, P) \to (N^2, Q) $ be a germ of a smooth map with a focus-focus singular point at $P$. Consider the Hessian of the map $\F$ at $P$. This is a symmetric bilinear form $$\diff^2_P\F \colon \T_PM^4 \times \T_PM^4 \to \T_QN^2.$$
\begin{proposition}
There exists a unique, up to sign, complex structure $J$ on $\T_QN^2$ such that $\diff^2_P\F $ becomes a complex bilinear form for a suitable choice of a complex structure on $\T_PM^4$.  The complex structures $\pm J$  with this property coincide with the ones constructed by pulling back the canonical complex structure on $\Complex$ by means of a normal chart.
\end{proposition}
\begin{proof}
Existence follows the fact that in suitable coordinates $\F$ becomes a holomorphic map (this also shows that the corresponding complex structures coincide with the ones constructed by means of a normal chart), while uniqueness can be demonstrated as follows. Using the normal form $(u,v) \mapsto uv$ of $\F$, one easily shows that there is unique, up to permutation of summands, decomposition 
$
\T_PM^4 = V_1 \oplus V_2,
$
where the spaces $V_1$ and $V_2$ are $2$-dimensional and maximally isotropic with respect to $\diff^2_P\F $. (Geometrically, $V_1$ and $V_2$ are tangent planes to the fiber of $\F$ at $P$.) Furthermore, for any $\xi \in V_1$, $\xi \neq 0$, the mapping
$$
D_\xi := \diff^2_P\F (\xi, *) \colon V_2 \to \T_QN^2
$$
is an isomorphism, so for any $\xi, \eta \in V_1$, $\xi \neq 0$, there is a well-defined operator
$$
R_{\xi\eta} := D_\eta \circ D_\xi^{-1} \colon \T_QN^2 \to \T_QN^2.
$$
Notice that if $\xi$ and $\eta$ are linearly independent, then the operator $R_{\xi\eta} $ cannot be scalar. (Otherwise $V_2$ is not maximal isotropic.) 
At the same time, if $\diff^2_P\F $ is a complex bilinear form, then $R_{\xi\eta} $ commutes with the complex structure on $\T_QN^2$. 
 But a non-scalar operator on a two-dimensional vector space commutes with at most two complex structures, which differ by sign.  So, there is at most two (in fact, exactly two by the existence part) complex structures on $\T_QN^2$ for which $\diff^2_P\F $ is complex bilinear, as desired.
\end{proof}
 \subsection{Classification of double pinched focus-focus singularities}\label{sec:dpinch}
 In this section we classify (generic) double pinched focus-focus singularities up to diffeomorphisms. First of all, for $n=2$, one can reformulate Proposition \ref{firstOrderInavriants} in the following way:
 
  \begin{corollary}\label{representative2p}
The orbit of any element $a z + b \bar z \in \Diff^1_0(\Complex)^+$ under the $\ADiff_0^1(\Complex)^2_\pm$-action has a unique representative of the form $z + \mu \bar z$, where $\mu \in \R$, $0 \leq \mu < 1$. The number $\mu$ is given by $\mu = |\frac{b}{a}|$.
 \end{corollary}

Thus, we obtain a function
$
 \mu \colon  \Diff_0(\Complex)_+\! \to [0,1)
$
 invariant under the action of $\ADiff_0(\Complex)^2_\pm$. Explicitly, this function reads
  \begin{align}
 \mu(\phi) =\left| \frac{{\partial \phi}/{\partial\bar z}(0)}{{\partial\phi}/{\partial   z}(0)}\right| \in  [0,1).
  \end{align}
For a double-pinched focus-focus singularity $\F$, we define $\mu(\F) := \mu(\phi)$  where $\phi := \phi_{1,2}$ is the corresponding orientation-preserving gluing map. 
\begin{corollary}
If double-pinched focus-focus singularities $\F$ and $\tilde \F$ are diffeomorphic, then $\mu(\F) = \mu(\tilde \F)$.
\end{corollary}
It turns out, that the converse result is also true, provided that $\mu \neq 0$. In other words, the space of double pinched focus-focus singularities is generically one-dimensional:
\begin{theorem}\label{thm:c1tocinfty} 
Assume that double-pinched focus-focus singularities $\F$ and $\tilde \F$ are such that $\mu(\F) = \mu(\tilde \F) \neq 0$. Then $\F$ and $\tilde \F$ are diffeomorphic. \end{theorem}
The proof is based on the corresponding algebraic statement:
\begin{lemma}\label{lemma:c1tocinfty} 
Assume that $\phi, \tilde \phi \in \Diff_0(\Complex)_+ $ are such that $\mu(\phi) = \mu(\tilde \phi) \neq 0$. Then $\phi$ and $\tilde \phi$ belong to the same orbit of $\ADiff_0(\Complex)^2_\pm$-action.
\end{lemma}
\begin{figure}[t]
\centerline{
\begin{tikzpicture}[thick, scale = 2]
\node at (0,0) () {\begin{tikzpicture}[thick, scale = 2]
\pgftransformxslant{1}
  \draw (0,0.1) -- (2,0.1) -- (2,0.9) -- (0,0.9) -- cycle;
     \fill (1,0.5) circle [radius=1.5pt];
     \pgftransformxscale{1.5}
\foreach \a in {1,2,3,4,5,6}{
\draw [thin] (0.66,0.5) to [bend right]+(\a*360/6: 0.3cm);
}
\end{tikzpicture}
};
\node at (0,0.7) () {\begin{tikzpicture}[thick, scale = 2]
\pgftransformxslant{1}
  \fill [white, draw = black] (0,0.1) -- (2,0.1) -- (2,0.9) -- (0,0.9) -- cycle;
\end{tikzpicture}
};
\draw [dotted] (1.1,0) -- (2.85,0);
\draw [very thick] (3,0) -- (3,1.33);
\draw [dotted] (1.1,0.7) -- (2.95,0.7);
\node at (3.2, 0) () {$0$};
\node at (3.2, 1.3) () {$1$};
\node at (3.15, 0.7) () {$\mu_0$};
\draw [very thick] (2.9,0) -- (3.1,0);
\draw  [very thick]  (2.9,1.3) to [bend left](3.1,1.3);
\end{tikzpicture}
}
\caption{ The $\mu$-invariant of double-pinched focus-focus singularities takes values in $[0,1)$. All singularities at the fixed level $\mu = \mu_0 \neq 0$ are pairwise diffeomorphic. The level $\mu = 0$ contains infinitely many diffeomorphism classes each containing the ``trivial'' singularity, i.e. the singularity whose gluing map is the identity, in its closure. }\label{fig:doubleff}
\end{figure}
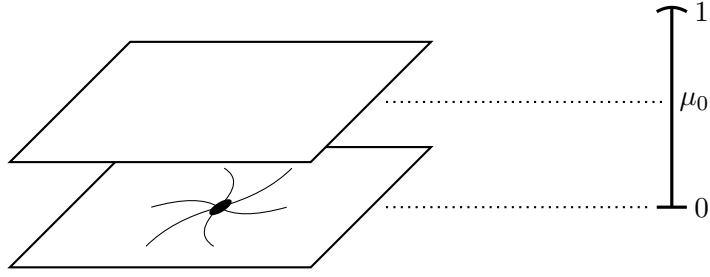
\begin{remark}
For $\mu(\phi) = \mu(\tilde \phi) = 0$ it is not necessarily true that $\phi$ and $\tilde \phi$ belong to the same orbit. For example, it is easy to see that all germs of the form $\phi = z + \bar z^k \in \Diff_0(\Complex)_+$ belong to different orbits. On the other hand, any $\infty$-jet $\phi \in \Diff_0^\infty(\Complex)_+ $ with $\mu(\phi) = 0$ contains the ``trivial'' jet $\phi_0 = z$ in its $\ADiff_0^\infty(\Complex)^2_\pm$-orbit closure. (Here we endow the space of $\infty$-jets with its natural Fr\'{e}chet topology.) Indeed, for any such $\phi$ and any $c \in \Complex^*$ we have
$$
\lim_{c \to \infty }cz \circ  \phi \circ c^{-1}z = z.
$$
%
This means that there exist no continuous invariants which distinguish between the orbits at the level $\mu = 0$, and the orbit space $ \Diff_0^\infty(\Complex)_+  \, / \, \ADiff_0^\infty(\Complex)^2_\pm$ (i.e. the space of double-pinched focus-focus singularities up to smooth equivalence) is non-Hausdorff. See Figure \ref{fig:doubleff}. 
\end{remark}
\begin{proof}[Proof of Lemma \ref{lemma:c1tocinfty}]
It suffices to show that if $\mu := \mu(\phi) \neq 0$, then $\phi$ lies in the same orbit as the linear function $z + \mu \bar z$. In other words, there exist liftable diffeomorphisms $\psi_1$, $\psi_2$ such that
\begin{align}\label{desired}
  ( z + \mu \bar z) \circ \psi_2 = \psi_1 \circ \phi.
\end{align}
We look for orientation-preserving liftable $\psi_1$, $\psi_2$. (Note that since the function $ z + \mu \bar z$ commutes with complex conjugation, existence of liftable $\psi_1$, $\psi_2$ satisfying~\eqref{desired} is equivalent to existence of orientation-preserving liftable $\psi_1$, $\psi_2$ with the same property.) By Theorem \ref{liftable} this means that $\psi_1(z) = zf(z)$, $\psi_2(z) = zg(z)$, where $f, g\in \Cont^\infty_0(\Complex , \Complex)$ are such that $f(0) \neq 0$ and $g(0) \neq 0$. In terms of the functions $f$, $g$, equation \eqref{desired} reads
\begin{equation}
 z g(z)+ \mu { \bar z   {\bar g(z)}}= \phi(z) f ( \phi(z)).
\end{equation}
We show that this equation has a solution $f, g \in \Cont^\infty_0(\Complex , \Complex)$  with $f(0) \neq 0$ and $g(0) \neq 0$. 
 Since $\phi$ is a diffeomorphism, this is equivalent to finding $g, h \in \Cont^\infty_0(\Complex , \Complex)$ with $g(0) \neq 0$ and $h(0) \neq 0$ such that
 \begin{equation}\label{desired3}
 z g(z)+ \mu {  \bar z   {\bar g(z)}}= \phi(z) h(z).
\end{equation}
Without loss of generality we can assume that
$
\phi(z) = z + \mu \bar z + \dots,
$
where  the dots denote higher order terms. (Indeed, by Corollary \ref{representative2p}, every orbit has a representative of this form.)
Then, since $\mu \neq 0$, one can write $\phi$ as
$$
\phi(z) = z + \mu \bar{ z} + zu(z) + \mu\bar zv(z),
$$
where $u, v  \in \Cont^\infty_0(\Complex , \Complex)$ are such that $u(0) = v(0) = 0$. Further, since $\phi$ is a diffeomorphism, the ideal generated in $\Cont^\infty_0(\Complex , \Complex)$ by $\phi$ and $\bar \phi$ is precisely $\{ w \in \Cont^\infty_0(\Complex , \Complex) \mid w(0) = 0\}$. This allows us to write $u$, $v$ as
$$
u(z) = \phi(z) u_1(z) + \bar \phi(z) u_2(z),\quad  v(z) = \phi(z) v_1(z) + \bar \phi(z) v_2(z),
$$ 
where $u_1, u_2, v_1, v_2 \in \Cont^\infty_0(\Complex , \Complex)$.
Then a straightforward substitution shows that the functions
$$
g := 1 + \phi \bar v_2 + \bar \phi u_2, \quad h := 1 + z(\bar v_2 - u_1) + \mu \bar z(\bar u_2 - v_1)
$$
solve \eqref{desired3}. Thus, Lemma \ref{lemma:c1tocinfty} is proved, and Theorem \ref{thm:c1tocinfty} follows.
\end{proof}

\begin{remark}\label{rem:npinchedstab}  We believe that a statement similar to that of Theorem~\ref{thm:c1tocinfty} is true for $n$-pinched singularities as well. Namely, we believe that for generic $n$-pinched focus-focus singularities the number of smooth invariants is finite, and $\Cont^\infty$-classification of such singularities can be reduced to $\Cont^k$-classification for certain $k = k(n)$. So far, we were not able to prove that conjecture. However, it is not hard to show that for generic $n$-pinched focus-focus singularities the number of \textit{finite order} invariants is finite.
More precisely, we have the following result.
\begin{proposition}  For the 
$\ADiff_0^k(\Complex)^n$ action on $\Diff_0^k(\Complex)^{n-1}$ given by
\eqref{gaugeAction}, the codimension of generic orbits is a bounded function of $k$.
\end{proposition}
\begin{proof}
Consider the $n$-tuple 
$$
\xi := (z + \mu_1 \bar z, \dots, z + \mu_{n-1} \bar z) \in \Diff_0^k(\Complex)^{n-1},
$$
where all $\mu_i$'s are non-zero and such that $\mu_i \neq \mu_j$, and $\mu_i \neq 1/\bar \mu_j$ for any $i,j$. Then, explicitly computing the stabilizer of $\xi$ under the $\ADiff_0^k(\Complex)^n$ action, one can show that its dimension is given by
$$
\dim \mathrm{Stab}\, \xi = \left[
\begin{aligned}
&\frac{1}{2}k(k+1), \mbox{ for } k < 2n-1, \\
&k^2 + (3-2n)k + (n-1)(2n -3), \mbox{ for } k \geq 2n-1.
\end{aligned}
\right.
$$
Therefore, the dimension of the orbit of $\xi$ is given by
\begin{align*}
\dim \mathrm{Orb}\, \xi &=  \dim \ADiff_0^k(\Complex)^n - \dim \mathrm{Stab}\, \xi \\&=  \left[
\begin{aligned}
&(n - \frac{1}{2})k(k+1), \mbox{ for } k < 2n-1, \\
&(n-1)k^2 + (3n -1)k - (n-1)(2n -3), \mbox{ for } k \geq 2n-1,
\end{aligned}
\right.
\end{align*}
while the codimension is given by
\begin{align*}
\codim \mathrm{Orb}\, \xi &=  \dim \Diff_0^k(\Complex)^{n-1}-\dim \mathrm{Orb}\, \xi \\&=  \left[
\begin{aligned}
&-\frac{1}{2}k^2 + (2n - \frac{5}{2})k, \mbox{ for } k < 2n-1, \\
& (n-1)(2n -3), \mbox{ for } k \geq 2n-1.
\end{aligned}
\right.
\end{align*}
So, for large $k$, there is an orbit of $\ADiff_0^k(\Complex)^n$ action on $\Diff_0^k(\Complex)^{n-1}$ whose codimension is $ (n-1)(2n -3)$. Therefore, the codimension of generic orbits is less or equal to this number, as desired.
\end{proof}
\end{remark}

\section{Symplectic focus-focus singularities}\label{sec:sympl}
\subsection{Any focus-focus singularity admits a symplectic structure}\label{sec:symp1}

The following result shows that \textit{smooth} classification for \textit{symplectic} focus-focus singularities is equivalent to that for smooth focus-focus singularities.
\begin{theorem}\label{thm:symplectic}
Any focus-focus singularity admits a symplectic structure which makes the corresponding fibration Lagrangian. \end{theorem}
To begin with, recall (see Proposition \ref{prop:norm}) that for any symplectic focus-focus point, there is a ``symplectic'' version of diagram \eqref{comDiag}, namely the top arrow $\Phi \colon (M^4, P) \to (\Complex^2, 0)$  is a symplectic map. (Here we endow $\Complex^2$ with the symplectic structure  $\mathrm{Re}\, (du\wedge dv)$.) We will refer to the corresponding bottom arrow $\phi \colon (N^2, Q) \to (\Complex, 0)$ as the \textit{symplectic normal chart}. Such a normal chart is unique up to multiplying its real and imaginary parts by $-1$ and adding a flat function to the real part (see Remark \ref{rem:uniqueness}). \par
Furthermore, in the case several focus-focus points on the fiber, the imaginary part of the symplectic normal charts $\phi_i \colon (N^2, Q) \to (\Complex, 0)$ agree up to sign: $\mathrm{Im}\, \phi_i=  \pm \mathrm{Im}\, \phi_j  $ (Theorem~\ref{th:focus}), and we can choose these charts in such a way that $\mathrm{Im}\, \phi_i=   \mathrm{Im}\, \phi_j  $. Then the corresponding \textit{symplectic gluing maps} $\phi_{i,j} := \phi_i \circ \phi_j^{-1}$ satisfy  $\mathrm{Im}\, \phi_{i,j}(z) =  \mathrm{Im}\, z  $. Conversely, any diffeomorphisms with this property can be realized as symplectic gluing maps:

\begin{proposition}\label{realization}
 Let $\{\phi_{i,j}  \in \Diff_0(\Complex)\}$ be a collection of diffeomorphisms satisfying \eqref{cocycle} and such that  $\mathrm{Im}\, \phi_{i,j}(z) =  \mathrm{Im}\, z  $  for every $i,j = 1, \dots, n$. Then there exists an $n$-pinched symplectic focus-focus singularity whose symplectic gluing maps are $\phi_{i,j}$'s.
\end{proposition}
\begin{proof}
The idea of the proof is to take a symplectic focus-focus singularity with identity gluing maps and then appropriately modify the Lagrangian fibration. Let $\F = (H, F) \colon (M^4, L) \to (\R^2, 0) $ be such a ``trivial'' singularity (see Remark \ref{rem:trivsing}). Since the gluing maps are trivial, one can assume that $(H, F)$ is a normal chart for each of the focus-focus points $P_1, \dots, P_n \in L$. Moreover, the function $F$ generates a global $S^1$-action.\par Now,  we change this Lagrangian fibration by modifying the function $H$. To that end, we take a cover   of a neighborhood of $L$ in $M^4$ by $S^1$-invariant open sets $U_1, \dots, U_n$ such that $P_i \in U_i$, and $P_i \notin \bar U_j$ for $j \neq i$. (Here $\bar U_j$ is the closure of $U_j$.) Let also $\{G_i\}$ be a partition of unity subordinate to the cover $\{U_i\}$. Without loss of generality, it can be assumed that the functions $G_i$ are invariant under the  $S^1$-action generated by $F$ (if not, we replace them by their averaged counterparts). We then define a new function $\tilde H \colon M^4 \to \R$ by
\begin{equation}\label{tildeh}
\tilde H := \sum_{i = 1}^n G_i \cdot \mathrm{Re}\, \phi_{1,i}^{-1}(H,F).
\end{equation}
The functions $\tilde H$ and $F$ Poisson-commute and thus give rise to a new Lagrangian fibration on $M^4$. As the initial fibration, the modified one has $L$ as its singular fiber of focus-focus type. Indeed, from \eqref{tildeh} we get
$$ 
(d \tilde H \wedge d F)\vert_L = J (d  H \wedge d F)\vert_L,$$
where the function $J$ is given by
$$
J := \sum_{i = 1}^n G_i \frac{\partial}{\partial H} \mathrm{Re}\, \phi_{1,i}^{-1}(H,F).
$$
Further, note that 
$$
\frac{\partial}{\partial H} \mathrm{Re}\, \phi_{1,i}^{-1}(H,F) > 0,
$$
since $\phi_{1,i}$ is orientation-preserving and has the form $(x,y)\mapsto (\dots, y)$, and also that $G_i \geq 0$, with at least one of $G_i$'s being strictly positive. Therefore, the function $J$ does not vanish, and singular points of the modified fibration which belong to $L$ are  the same as for the initial fibration. Furthermore, these points are of focus-focus type, because the two fibrations coincide near each of the singular points. So, the fibration defined by $\tilde H$ and $F$ has $L$ as its singular fiber of focus-focus type. Furthermore, it is easy to see that the gluing maps for the new fibration are $\phi_{i,j}$'s. Thus, the proposition is proved.
\end{proof}
\begin{remark}\label{rem:trivsing}
A symplectic focus-focus singularity with trivial gluing maps can be constructed, for instance, as an $n$-fold covering of a focus-focus singularity with $1$ pinch point. Indeed, for a sufficiently small neighborhood $U$ of a focus-focus singularity with $1$ pinch point, its fundamental group is isomorphic to $\Z$. Therefore, one can construct an $n$-fold covering $\pi \colon V \to U$ corresponding to the index $n$ subgroup $n\Z \subset  \pi_1(U)$. Further, one lifts the symplectic structure and the Lagrangian fibration from $U$ to $V$ using the projection $\pi$. Clearly, the so-obtained fibration on $V$ has a focus-focus fiber with $n$ pinch points as the $\pi$-preimage of the singular fiber in $U$. Furthermore, as normal charts for the focus-focus fibration on $V$ one can take the normal chart for the fibration on $U$. Hence, all gluing maps for the focus-focus fibration on $V$ are trivial for suitable choice of normal charts, as desired.
\end{remark}
\begin{remark}\label{rem:family}
Although there exist different approaches to the proof of Proposition \ref{realization} (see e.g. \cite[Section 7]{SanFF}), the advantage of our approach is that it allows one to construct focus-focus singularities with all possible gluing maps on one and the same symplectic manifold. Moreover, given a family of germs  $\{\phi_{i,j}^t\}$ depending smoothly on a parameter $t \in \R$, our construction produces a smooth family of singularities. This will be important in Section \ref{sec:example}.
\end{remark}
\begin{lemma}\label{incl}
    For any germ $f\colon(\mathbb \Complex,0) \to (\mathbb R,0)$ such that $d f(0) \neq 0$ there exists a liftable germ $\psi \in \ADiff_0(\Complex)$ such that $\mathrm{Im}\, \psi = f$.
\end{lemma}
\begin{proof}
  Let $z = x+\i y$ be the coordinate in $\Complex$. Write $f$ as
$
        f=xv(z)+yu(z)
$, where  $u,v \in  \Cont^\infty_0(\Complex , \R)$,
    and set
$$
        \psi(z) :=xu(z)-yv(z) + \i f(z).
$$
Then, from the condition $d f(0) \neq 0$ it follows that $\psi$ is a diffeomorphism. Furthermore, 
$$
\psi = xu-yv + \i (xv+yu) = (x+\i y)(u+\i v),
$$
so $\psi$ is liftable, as desired.
\end{proof}
\begin{proof}[Proof of Theorem \ref{thm:symplectic}] The statement of the theorem can be reformulated as follows: Any smooth focus-focus singularity (in the sense of Definition~\ref{def:focussmooth}) is diffeomorphic to a symplectic one. Thanks to Theorem \ref{thm:thm1} and Proposition \ref{realization}, this is equivalent to saying that, for any tuple $(\phi_{1,2},  \dots, \phi_{1,n}) \in \Diff_0(\Complex)^{n-1}$, its orbit  under the  $\ADiff_0(\Complex)^n$ action has a representative of the form $( \tilde \phi_{1,2},  \dots,  \tilde \phi_{1,n}) $ where $\mathrm{Im}\,  \tilde \phi_{1, i}(z) = \mathrm{Im}\, z$ for every $i = 2, \dots, n$.  
 To prove the latter, take any $\phi_{1,2},  \dots, \phi_{1,n} \in \Diff_0(\Complex)$. Then, by Lemma \ref{incl}, there exist liftable $\psi_2, \dots, \psi_n \in \ADiff_0(\Complex)$ such that $\mathrm{Im}\, \psi_i =  \mathrm{Im}\,\phi_{1,i}$. Notice that $$\mathrm{Im}\, \psi_i( \phi_{1,i}^{-1}(z)) = \mathrm{Im}\,  \phi_{1,i}( \phi_{1,i}^{-1}(z)) =  \mathrm{Im}\,z. $$
Therefore, the inverse map $ \tilde \phi_{1,i} := (\psi_i \circ \phi_{1,i}^{-1})^{-1} =  \phi_{1,i} \circ \psi_i^{-1}$ also satisfies  $\mathrm{Im}\,  \tilde \phi_{1, i}(z) = \mathrm{Im}\, z$, as desired.
\end{proof}

\subsection{First order invariants in terms of eigenvalues}

Let $\F$ be a symplectic focus-focus singularity, and let $H$ be a generic function constant on the fibers of $\F$. (Here generic means that $\partial H / \partial( \mathrm{Re}\,\phi_i) \neq 0$, where $\phi_i$ is  the symplectic normal chart corresponding to the singular point $P_i$.)
In this section we express first order invariants of $\F$ in terms of eigenvalues of the corresponding Hamiltonian vector field $\sgrad H$ linearized at singular points. \par
Let $A_i \colon T_{P_i}M^4 \to  T_{P_i}M^4$ be the linearization of $\sgrad H$ at the singular point $P_i$. Then the eigenvalues of $A_i$ form a quadruple symmetric with respect to the real and imaginary axes. We choose one eigenvalue out of the quadruple in the following way. Let  $\phi_i$ be the symplectic normal chart corresponding to the point $P_i$. Then we have $$H = a_i\mathrm{Re}\,\phi_i + b_i\mathrm{Im}\,\phi_i +\dots$$
(where dots denote higher order terms), and eigenvalues of $A_i$ are exactly $\pm a_i \pm \i b_i$. Then, as a preferred eigenvalue, we choose $\lambda_i := a_i + \i b_i$. This gives a canonical way to choose  eigenvalues $\lambda_1, \dots, \lambda_n$ (one for each point $P_i$), up to simultaneous complex conjugation or simultaneous multiplication by $-1$.  (Here we assume that the normal charts are chosen in such a way that their orientations agree.)

\begin{proposition}\label{prop3.4}
Assume that $\lambda_i$ is the eigenvalue of the linearization of $\sgrad H$ at the singular point $P_i$, chosen as described above. Then the first order invariants $\mu_2, \dots, \mu_n$ are given by
   \begin{align}\label{evFormula}
        \mu_i = \frac{ \lambda_i - \lambda_1}{ { \lambda}_i + \bar \lambda_1}.
    \end{align}
\end{proposition}
\begin{proof}
First of all notice that if $H$ is replaced by another generic Hamiltonian $\tilde H$, then the corresponding eigenvalues change as $ \lambda_i \mapsto a\lambda_i + b \i $, where $a,b \in \R$ are the same for all $i$'s. Therefore, the expression on the right-hand side of \eqref{evFormula} does not depend on the choice of the Hamiltonian.\par 

Let  $\phi_i$ be the symplectic normal chart corresponding to the point $P_i$. To compute the invariants $\mu_i$, we set $H := \mathrm{Re}\,\phi_1$. Writing $H$ in terms of the normal chart $\phi_i$, we get $$H = a_i\mathrm{Re}\,\phi_i + b_i\mathrm{Im}\,\phi_i + \dots$$ Since $H =\mathrm{Re}\,\phi_1$, and $\mathrm{Im}\,\phi_1 = \mathrm{Im}\,\phi_i$, the gluing map $\phi_{1,i} = \phi_1 \circ \phi_i^{-1}$ has the form $$x + \i y \mapsto a_i x + b_i y + \i y + \dots,$$
so
$$
\mu_i = \frac{{\partial \phi_{1,i}}/{\partial\bar z}(0)}{\phantom{\int}\overline{{{\partial\phi_{1,i}}}/{\partial   z}}(0)\phantom{\int}} = \frac{a_i + b_i\i  - 1}{a_i + b_i\i  + 1} =  \frac{ \lambda_i - \lambda_1}{ { \lambda}_i + \bar \lambda_1},
$$
where we used that $\lambda_i = a_i + b_i\i $ and $\lambda_1 = 1$ due to the choice of $H$.
\end{proof}

\begin{remark} The procedure of choosing one eigenvalue from a quadruple can also be performed without knowing the normal charts. First of all, one should choose $\lambda_i$'s such that the sign of $\mathrm{Re}\, \lambda_i$ is the same for all $i = 1, \dots, n$. (One has $\mathrm{Re}\, \lambda_i \neq 0$ since $H$ is generic.) Furthermore, one can distinguish between $\lambda_i$ and $\bar \lambda_i$ in the following way. Instead of a particular Hamiltonian $H$, consider the whole $2$-dimensional family of commuting Hamiltonians, $a H + bF$. Then the corresponding linearization at $P_i$ depends on the parameters $a,b$, and  $\lambda_i$ becomes a bilinear function of $a,b$: $\lambda_i = \lambda_i(a,b)$. Further, according to Theorem \ref{th:focus}, our symplectic focus-focus singularity admits a global $S^1$-action. Although the generator of this action does not have to be of the form $a H + bF$, it is of such form up to higher order terms. So, there exist $a,b \in \R$ such that  $\lambda_i(a,b) = \pm \i$ for each $i = 1, \dots, n$. Then we choose $\lambda_i$ in such a way that $\lambda_i(a,b) = \i$ for any $i$ (or $-\i$ for any $i$). 
\end{remark}

\section{An obstruction to smooth almost direct product decomposition}\label{sec:example}


In this section we construct a Lagrangian fibration in dimension $6$ with a rank $1$ focus-focus singularity  which is homeomorphic to the direct product of a rank $0 $ focus-focus singularity and a trivial fibration, but not diffeomorphic to it. This disproves  a conjecture stated by Zung in  \cite{AL} which says that any non-degenerate singularity can be (semilocally, i.e. in the neighborhood of the singular fiber) {\it smoothly} decomposed into an almost direct product of elementary bricks of four types:  regular, elliptic, hyperbolic, and focus-focus..\par
Our construction is as follows. Take a family $\F_t$ of double-pinched symplectic focus-focus singularities on $(M^4, \omega)$ depending on the parameter $t \in (a,b) \subset \R$ in such a way that the $\mu$-invariant defined in Section \ref{sec:dpinch} varies within the family: $\mu = \mu(t)$. (The existence of such a family follows from Proposition \ref{realization}, see also Remark \ref{rem:family}.)  Such a family gives rise to a Lagrangian fibration on $M^6 := M^4 \times (a,b) \times S^1$ endowed with the symplectic structure $\omega + \diff t \wedge \diff \phi$, where $\phi$ is the coordinate on $S^1$. The corresponding moment map $\tilde \F \colon M^6 \to \R^3$ is given by $\tilde \F(x, t, \phi) = (\F,t) $. This fibration has a focus-focus singularity of rank $1$, with two critical circles on each fiber. By construction, it is homeomorphic to the direct product of a double-pinched rank $0$ focus-focus singularity and a regular foliation of an annulus by concentric circles.
\begin{proposition}
This singularity is not diffeomorphic to an (almost) direct product.
\end{proposition}
\begin{proof}
The idea of the proof is to show that  the $\mu$-invariant is well-defined for rank $1$ focus-focus singularities with two critical circles on the fiber. In this case, this invariant is no longer a number, but a function on the set of critical values of the moment map. (The latter is a smooth curve $\Sigma \subset \R^3$.) The definition of this invariant is compatible with the rank $0$ case in the following sense: if a rank $1$ singularity is diffeomorphic to a direct product, then its  $\mu$-invariant is a constant function equal to the $\mu$-invariant of the corresponding rank $0$ singularity.\par
To define this invariant, we repeat the construction of Section \ref{sec:foics}. Namely, consider the Hessian of the moment map $\F$ at a rank one focus-focus point $P$. This is now a bilinear map
$$\diff^2_P\F \colon \Ker \diff_P \F \times \Ker \diff_P\F \to \Coker \diff_P\F,$$
where $ \Coker \diff_P\F := \T_{\F(P)}\R^3 \, / \, \mathrm{Im}\,  \diff_P\F$ is the cokernel of the differential of $\F$ at $P$.
As in Section \ref{sec:foics}, one shows that there exists a unique, up to sign, complex structure on $ \Coker \diff_P\F$ which lifts to a complex structure on $ \Ker \diff_P \F$ in such a way that $\diff^2_P\F$ becomes a complex bilinear map. Moreover, this complex structure does not depend on the choice of the point $P$ on the critical orbit. Indeed, for any other point $\tilde P$ on the same critical orbit, there is a fiberwise diffeomorphism (in fact, even a symplectomorphism) taking $P$ to $\tilde P$. But since our construction is invariant under diffeomorphisms, it follows that the corresponding complex structures on  $ \Coker \diff_P\F =  \Coker \diff_{\tilde P}\F$ are the same. \par
Now, considering both critical orbits on the same fiber, we get two complex structures on the space $\Coker \diff\F$. Also notice that the latter can be viewed as a fiber in the normal bundle $\N\Sigma$ to the set of critical values $\Sigma \subset \R^3$. (Indeed, for non-degenerate singularities, the image of the differential of the moment map is exactly the tangent space to the set $\Sigma$ of critical values.) Repeating the construction for every singular value $Q \in \Sigma$, we get two complex structures $J_1, J_2$ in the normal bundle $\N\Sigma$, and hence a  function $\tr(J_2J_1^{-1}) \colon \Sigma \to \R$ which is invariant under diffeomorphisms. (While this is not exactly the $\mu$-invariant, those invariants are functions of each other given by formula~\eqref{trmu}.)\par
It is clear that the so-constructed invariant should be a constant function on $\Sigma$ for direct-product-type singularities. Moreover, this invariant does not change under covering maps, so it is also constant for almost direct products. On the other hand, the $\mu$-invariant of the singularity constructed above is a non-trivial function. Therefore, this singularity is not diffeomorphic to any almost direct product.
\end{proof}

\bibliographystyle{plain}
\bibliography{sff}


\end{document}